\documentclass[oneside,english]{amsart}
\usepackage[T1]{fontenc}
\usepackage[latin9]{inputenc}
\usepackage{amsthm}
\usepackage{amstext}
\usepackage{amssymb}
\usepackage{esint}
\usepackage{graphicx}
\usepackage{enumerate}
\usepackage{amscd}

\usepackage{graphicx}
\graphicspath{{noiseimages/}}

\makeatletter
\newtheorem{theorem}{Theorem}[section]

\numberwithin{figure}{section}
\theoremstyle{definition}

\theoremstyle{remark}
\newtheorem{remark}[theorem]{Remark}

\numberwithin{equation}{section}
\makeatother

\usepackage{babel}

	\DeclareMathOperator{\loc}{loc}

    \DeclareMathOperator{\Imag}{Im}
    \DeclareMathOperator{\Real}{Re}

\begin{document}

\title[principal frequencies and extension operators]{Principal frequencies of free non-homogeneous membranes and Sobolev extension operators}

\author{V.~Gol'dshtein, V.~Pchelintsev, A.~Ukhlov}

\begin{abstract}
Using the quasiconformal mappings theory and Sobolev extension operators, we obtain estimates of principal frequencies of free non-homogeneous membranes. The suggested approach is based on connections between divergence form elliptic operators and quasiconformal mappings. Free non-homogeneous membranes we consider as circular membranes in the quasiconformal geometry associated with non-homogeneity of membranes. As a consequence we get a connection between principal frequencies of free membranes and the smallest-circle problem (initially suggested by J.~J. Sylvester in 1857).
\end{abstract}
\maketitle
\footnotetext{\textbf{Key words and phrases:} elliptic equations, Sobolev spaces, extension operators.}
\footnotetext{\textbf{2010
Mathematics Subject Classification:} 35P15, 46E35, 30C65.}

\section{Introduction}

In the previous article \cite{GPU2020S} we studied the connections between quasiconformal mappings, the corresponding extension operator of Sobolev spaces and the Neumann eigenvalues of the Laplacian (principal frequencies of free homogeneous membranes). This article focused on estimates of principal frequencies of free non-homogeneous membranes (see, for example, \cite{Gr75}) which are circular membranes in the quasiconformal geometry associated with the divergence form elliptic operators. Hence the principal frequencies of free non-homogeneous membranes are the Neumann eigenvalues of the divergence form elliptic operators
\begin{equation*}\label{EllDivOper}
L_{A}=-\textrm{div} [A(z) \nabla u(z)], \quad z=(x,y)\in \Omega ,
\end{equation*}
are defined in bounded simply connected domains $\Omega \subset \mathbb R^2$.
We assume that the matrix $A$ is defined a.~e. in $\mathbb{R}^2$ and satisfies the following regularity conditions (see, for example, \cite{AIM,BGMR}):

\noindent
1) The matrix $A$ belongs to the class of all $2 \times 2$ symmetric matrix functions $A(w)=\left\{a_{kl}(w)\right\}$   with measurable entries $a_{kl}(w)$ defined in $\mathbb{R}^2$ and under additional condition $\textrm{det} A=1$ a.e. in $\mathbb R^2$.

\noindent
2) The matrix $A$ satisfies to the uniform ellipticity condition: there exists $1\leq K< \infty$ such that the inequality
\begin{equation}\label{UEC}
\frac{1}{K}|\xi|^2 \leq \left\langle A(w) \xi, \xi \right\rangle \leq K |\xi|^2, \,\,\, \text{a.e. in}\,\,\, \mathbb R^2,
\end{equation}
holds for every $\xi \in \mathbb C$, which can be written as
\begin{equation*}
|\mu(z)|\leq \frac{K-1}{K+1},\,\,\, \text{a.e. in}\,\,\, \Omega.
\end{equation*}

Under these conditions the matrix $A$ induces a quasiconformal mapping $\varphi_A:\mathbb R^2\to\mathbb R^2$ which is the solution to the Beltrami equation with the complex dilatation $\mu$ defined by the matrix $A$:
\begin{equation*}
\mu(z)=\frac{a_{22}(z)-a_{11}(z)-2ia_{12}(z)}{\det(I+A(z))},\quad I= \begin{pmatrix} 1 & 0 \\ 0 & 1 \end{pmatrix}.
\end{equation*}
Recall that such mappings are called as $A$-quasiconformal mappings \cite{GPU2020}.

In the present work we study the Neumann eigenvalues in an $A$-quasidisc $\mathbb{D}_A$ that can be considered as the unit disc in
the Riemannian geometry, generated by the matrix $A$ and/or by the $A$-quasiconformal mappings associated with $A$. Namely, $\mathbb{D}_A:=\varphi_A^{-1}(\mathbb D)$, where $\mathbb D\subset\mathbb R^2$ is the unit disc.  Since the quasiconformal mappings $\varphi_A$ are defined by the composition factor with conformal mappings we can assume that the area of $|\mathbb{D}_A|=|\mathbb D|=\pi$.

\begin{remark}
Recall that a quasidisc (an Alhfors domain) is the image of the unit disc $\mathbb{D}$ under a quasiconformal mapping $\varphi:\mathbb R^2\to\mathbb R^2$. Since any $A$-quasiconformal mapping $\varphi_A:\mathbb R^2\to\mathbb R^2$ is a quasiconformal one then an $A$-quasidisc is a quasidisc and satisfy to Ahlfors condition
\cite{Ahl66} also. We use the concept of the $A$-quasidisc because its relation with the matrix $A$.
\end{remark}

Since the matrix $A$ satisfies to the uniform ellipticity condition, the functional
\[
F_A(\nabla u,\Omega)=\left(\iint\limits_\Omega \left\langle A(z)\nabla u(z),\nabla u(z)\right\rangle\,dxdy \right)^{\frac{1}{2}},
\]
defines an equivalent seminorm in the Dirichlet space $L^{1,2}(\Omega)$ \cite{GPU2020}.
We denote the Dirichlet space $L^{1,2}(\Omega)$ with this new seminorm as  $L_{A}^{1,2}(\Omega)$. It is well known that quasiconformal mappings $\varphi:\Omega \to \widetilde{\Omega}$ between domains $\Omega,\widetilde{\Omega}\subset\mathbb R^2$ induces an isomorphism between  $L^{1,2}(\widetilde{\Omega})$ and $L^{1,2}(\Omega)$ according to the standard chain rule $\varphi^{\ast}(f)=f\circ\varphi $ \cite{VG75}. In \cite{GPU2020} we have refined this result for the case of the spaces $L^{1,2}_A(\widetilde{\Omega})$ and $L^{1,2}(\Omega)$ and on this basis we obtained estimates of semi-stability of the Neumann eigenvalues under quasiconformal mappings.

%In the present work we give sharp estimates of the Neumann eigenvalues, based on estimates for the norms of the extension operators \cite{GPU2020S} for the Sobolev spaces $L^{1,2}_A(\mathbb{D}_A)$.

In the present work we establish a connection between the Neumann eigenvalue problem of elliptic operators in the divergence form and the smallest-circle problem.
The smallest-circle problem is a mathematical problem of computing the smallest circle that contains all of a given set of points in the Euclidean plane.

This connection will be given in
the form of sharp lower estimates for the first non-trivial Neumann eigenvalues via the radii of the smallest enclosing discs  (Theorem A).

Recall that according to the Min--Max Principle (see, for example, \cite{D95,M}) the first non-trivial Neumann eigenvalue $\mu_1(A,\Omega)$ of the divergence form elliptic operator $L_{A}$ can be represented as
$$
\mu_1(A,\Omega)=\min\left\{\frac{\|u \mid L^{1,2}_{A}(\Omega)\|^2}{\|u \mid L^{2}(\Omega)\|^2}:
u \in W^{1,2}_{A}(\Omega) \setminus \{0\},\,\, \iint\limits _{\Omega}u\, dxdy=0 \right\}.
$$
Hence, the value $\mu_1(A,\Omega)^{-\frac{1}{2}}$ is the best constant $B_{2,2}(A,\Omega)$ in the following Poincar\'e type inequality
$$
\inf\limits _{c \in \mathbb R} \|u-c \mid L^2(\Omega)\| \leq B_{2,2}(A,\Omega) \|u \mid L^{1,2}_{A}(\Omega)\|, \quad u \in W^{1,2}_{A}(\Omega).
$$
Let us remark that by the uniform ellipticity condition $B_{2,2}(A,\Omega)\leq \sqrt{K} B_{2,2}(I,\Omega)$.

The main result of the article gives a connection between principal frequencies of free non-homogeneous membranes and the smallest-circle problem. Namely:
\vskip 0.3cm
\noindent
{\bf Theorem A.}
\textit{Let $\mathbb D_A$ be an $A$-quasidisc. Then the following inequality holds
\begin{equation*}
\mu_1(A,\mathbb{D}_A) \geq  \frac{1}{4K} \cdot \left(\frac{j'_{1,1}}{R_{\mathbb{D}_A}}\right)^2,
\end{equation*}
where $K$ is the ellipticity constant, $R_{\mathbb{D}_A}=\max\limits_{\mathbb{D}(0,1)}|\varphi^{-1}(x)-\varphi^{-1}(0)|$ and $j'_{1,1}\approx 1.84118$ denotes the first positive zero of the derivative of the Bessel function $J_1$.
}

\vskip 0.3cm

In this work we give the estimates of exponential type for the radius $R_{\mathbb{D}_A}$ of the smallest disc that contains $\mathbb{D}_A$ (see Theorem~\ref{sil}). Such type estimates represent the partial solution of the smallest-circle problem (initially suggested by J.~J. Sylvester in 1857).

The exponential type estimates of  the radius $R_{\mathbb{D}_A}$ leads to the following estimates of principal frequencies of free non-homogeneous membranes in quasiconformal terms:

\vskip 0.3cm

\noindent
{\bf Theorem B.}
\textit{Let $\mathbb D_A$ be an $A$-quasidisc. Then the following inequality holds
\begin{equation*}
\mu_1(A,\mathbb{D}_A) \geq \frac{ 64 (j'_{1,1})^2}{K \exp \{2\pi K\}},
\end{equation*}
where $K$ is the ellipticity constant and $j'_{1,1}\approx 1.84118$ denotes the first positive zero of the derivative of the Bessel function $J_1$.
}

\vskip 0.3cm

Let us remind that a simply connected domain $\Omega\subset\mathbb R^2$ is called \textit{a $L^1_2$-extension domain} if there exists a continuous extension operator
$$
E:L^{1,2}_{A}(\Omega) \to L_A^{1,2}(\mathbb R^2).
$$

This extension operator does not depend on the matrix $A$, but only its norm depends on the matrix $A$. The Ahlfors domains \cite{Ahl66} represent an important subclass of $2$-extension domains.

The proof of Theorem A is based on the following general result:
\vskip 0.3cm
\noindent
{\bf Theorem C.}
\textit{ Let $\Omega$ and $\widetilde{\Omega}$ be Ahlfors domains (quasidisc). If $\widetilde{\Omega}\supset \Omega$ then
$$
\mu_1(A,\Omega) \geq \frac{\mu_1(A,\widetilde{\Omega})}{\|E_{\Omega}\|^2},
$$
where $\|E_{\Omega}\|$ denotes the norm of the linear continuous extension operator
$$
E_{\Omega}: L^{1,2}_A(\Omega)\to L_A^{1,2}(\widetilde{\Omega}).
$$
}

The suggested method is based on estimating the norm of extension operators on Sobolev spaces with the divergence norm:
\[
\|Eu \mid L_A^{1,2}(\mathbb R^2)\| \leq 2 \cdot \|u \mid L^{1,2}_{A}(\mathbb D_A)\|,
\,\,\,u\in L^{1,2}_{A}(\mathbb D_A).
\]

\begin{remark}
{\it The extension problem of quasiconformal planar mappings was studied in details by L.Ahlfors \cite{Ahl66}. He obtained a geometric description of such domains (so-called Ahlfors domains or quasidiscs)}.
\end{remark}

\section{Extension operators and $A$-quasiconformal mappings}

\subsection{Sobolev spaces}
Let $\Omega$ be a simply connected domain in $\mathbb{R}^2$.
For any $1\leq p<\infty$ we consider the Lebesgue space $L^p(\Omega)$ of measurable functions $u: \Omega \to \mathbb{R}$ equipped with the following norm:
\[
\|u\mid L^p(\Omega)\|=\biggr(\iint\limits _{\Omega}|u(z)|^{p}\, dxdy\biggr)^{\frac{1}{p}}<\infty.
\]

The Sobolev space $W^1_p(\Omega)$, $1\leq p<\infty$, is defined
as a Banach space of locally integrable weakly differentiable functions
$u:\Omega\to\mathbb{R}$ equipped with the following norm:
\[
\|u\mid W^{1,p}(\Omega)\|=\|u\,|\,L^{p}(\Omega)\|+\|\nabla u\mid L^{p}(\Omega)\|.
\]

The seminormed Sobolev space $L^{1,p}(\Omega)$, $1\leq p< \infty$,
is the space of all locally integrable weakly differentiable functions $u:\Omega\to\mathbb{R}$ equipped
with the following seminorm:
\[
\|u\mid L^{1,p}(\Omega)\|=\|\nabla u\mid L^p(\Omega)\|, \,\, 1\leq p<\infty.
\]

We also need a weighted seminormed Sobolev space $L_{A}^{1,2}(\Omega)$ (associated with the matrix $A$), defined
as the space of all locally integrable weakly differentiable functions $u:\Omega\to\mathbb{R}$ equipped with the following seminorm:
\[
\|u\mid L_{A}^{1,2}(\Omega)\|=\left(\iint\limits_\Omega \left\langle A(z)\nabla u(z),\nabla u(z)\right\rangle\,dxdy \right)^{\frac{1}{2}}.
\]

The corresponding  Sobolev space $W^{1,2}_{A}(\Omega)$ is defined
as the normed space of all locally integrable weakly differentiable functions
$u:\Omega\to\mathbb{R}$ equipped with the following norm:
\[
\|u\mid W^{1,2}_{A}(\Omega)\|=\|u\,|\,L^{2}(\Omega)\|+\|u\mid L^{1,2}_{A}(\Omega)\|.
\]

Recall that a continuous operator
\[
E:L^{1,2}_{A}(\Omega) \to L_A^{1,2}(\mathbb R^2)
\]
satisfying the conditions
$$
Eu|_{\Omega}=u \quad \text{and} \quad \|E\| :=\sup\limits_{u \in L^{1,2}_{A}(\Omega)} \frac{\|Eu\|_{L^{1,2}_A(\mathbb R^2)}}{\|u\|_{L^{1,p}_{A}(\Omega)}}<\infty
$$
is called an continuous extension operator on Sobolev spaces $L^{1,2}_A$.

We say that $\Omega\subset\mathbb R^2$ is \textit{a Sobolev $L^{1,2}_{A}$-extension domain} if there exists a continuous extension operator
$$
E:L^{1,2}_{A}(\Omega) \to L_A^{1,2}(\mathbb R^2).
$$

\begin{remark}
The spaces $L^{1,2}(\Omega)$ and $L^{1,2}_{A}(\Omega)$, $\Omega\subset\mathbb R^2$, both are spaces of weak differentiable functions with square integrable first derivatives equipped with equivalent seminorms. Hence Sobolev $L^{1,2}$-extension domains and Sobolev $L^{1,2}_{A}$-extension domains coincide and are Ahlfors domains (quasidics) \cite{VGL79}.
\end{remark}

It is well known that existence of an extension operator from $L^{k,p}(\Omega)$ to $L^{k,p}(\mathbb R^n)$, $k\in\mathbb N$, $n\geq2$, depends on the geometry of the domain $\Omega$. In the case of Lipschitz domains, Calder\'on \cite{Cal61} constructed an extension operator on $L^{k,p}(\Omega)$ for $1 < p < \infty$, $k\in\mathbb N$. Stein \cite{S70} extended Calder\'on's result to the endpoints $p=1, \infty$.
Jones \cite{Jon81} introduced an extension operator for locally uniform domains. This class is much broader then previous classes and includes examples of domains with highly non-rectifiable boundaries.

In \cite{VGL79} necessary and sufficient conditions were obtained for $L^{1,2}$-extensions operators from planar simply connected domains in terms of the quasiconformal geometry of domains. On this base in \cite{GPU2020S} were obtained sharp estimates of norms of the $L^{1,2}$-extension operators from Ahlfors domains.
Necessary and sufficient conditions for $L^{1,p}$-extension operators from planar simply connected domains were obtained in \cite{Shv16, Shv17} for $p>2$ and in the case $1<p<2$ in \cite{KRZ19}. Note, that in case $p\neq 2$ estimates of extension operators norms are unknown.

\subsection{Quasiconformal mappings associated with the matrix $A$}
Recall that a homeomorphism $\varphi: \Omega \to \widetilde{\Omega}$, where domains $\Omega,\, \widetilde{\Omega} \subset\mathbb C$, is called a $K$-quasiconformal mapping if $\varphi\in W^{1,2}_{\loc}({\Omega})$ and there exists a constant $1\leq K<\infty$ such that
$$
|D\varphi(z)|^2\leq K |J(z,\varphi)|\,\,\text{for almost all}\,\,z \in \Omega.
$$

Let us give a construction of $A$-quasiconformal mappings connected with the $A$-divergent form elliptic operators \cite{AIM}.
We suppose that matrix functions $A(z)=\left\{a_{kl}(z)\right\}$ belong to the class of all $2 \times 2$ symmetric matrix functions $A(w)=\left\{a_{kl}(w)\right\}$   with measurable entries defined in $\mathbb{R}^2$, such that $\textrm{det} A=1$ a.e. in $\mathbb R^2$, and are satisfied to the uniform ellipticity condition (\ref{UEC}).

The basic idea is that every positive quadratic form
\[
ds^2=a_{11}(x,y)dx^2+2a_{12}(x,y)dxdy+a_{22}(x,y)dy^2
\]
defined in a planar domain $\Omega$ can be reduced, by means of a quasiconformal change of variables, to the canonical form
\[
ds^2=\Lambda(du^2+dv^2),\,\, \Lambda\neq 0,\,\, \text{a.e. in}\,\, \widetilde{\Omega},
\]
given that $a_{11}a_{22}-a^2_{12} \geq \kappa_0>0$, $a_{11}>0$, almost everywhere in $\Omega$ \cite{Ahl66,BGMR}. We will use the complex variable  $z=x+iy$ that is more convenient for this study of the operator $\textrm{div} [A(z) \nabla f(z)]$.

Let $\xi(z)=\Real \varphi(z)$ be a real part of a quasiconformal mapping $\varphi(z)=\xi(z)+i \eta(z)$, which satisfies to the Beltrami equation:
\begin{equation}\label{BelEq}
\varphi_{\overline{z}}(z)=\mu(z) \varphi_{z}(z),\,\,\, \text{a.e. in}\,\,\, \Omega,
\end{equation}
where
$$
\varphi_{z}=\frac{1}{2}\left(\frac{\partial \varphi}{\partial x}-i\frac{\partial \varphi}{\partial y}\right) \quad \text{and} \quad
\varphi_{\overline{z}}=\frac{1}{2}\left(\frac{\partial \varphi}{\partial x}+i\frac{\partial \varphi}{\partial y}\right),
$$
with the complex dilatation $\mu(z)$ is given by
\begin{equation}\label{ComDil}
\mu(z)=\frac{a_{22}(z)-a_{11}(z)-2ia_{12}(z)}{\det(I+A(z))},\quad I= \begin{pmatrix} 1 & 0 \\ 0 & 1 \end{pmatrix}.
\end{equation}
We call this quasiconformal mapping (with the complex dilatation $\mu$ defined by (\ref{ComDil})) as an $A$-quasiconformal mapping.

Note that the uniform ellipticity condition \eqref{UEC} can be written as
\begin{equation*}\label{OVCE}
|\mu(z)|\leq \frac{K-1}{K+1},\,\,\, \text{a.e. in}\,\,\, \Omega.
\end{equation*}

Conversely we can obtain from the complex dilatation \eqref{ComDil} (see, for example, \cite{AIM}, p. 412) the matrix $A$ by the following way :
\begin{equation*}\label{Matrix-F}
A(z)= \begin{pmatrix} \frac{|1-\mu|^2}{1-|\mu|^2} & \frac{-2 \Imag \mu}{1-|\mu|^2} \\ \frac{-2 \Imag \mu}{1-|\mu|^2} &  \frac{|1+\mu|^2}{1-|\mu|^2} \end{pmatrix},\,\,\, \text{a.e. in}\,\,\, \Omega.
\end{equation*}

Hence, the given matrix  $A$ one produced, by \eqref{ComDil}, the complex dilatation $\mu(z)$, for which, in turn, the Beltrami equation \eqref{BelEq} induces a quasiconformal homeomorphism
$\varphi_A:\Omega \to \widetilde{\Omega}$ as its solution, by the Riemann measurable mapping theorem (see, for example, \cite{Ahl66,AIM}). We will say that the matrix function $A$ induces the corresponding $A$-quasiconformal homeomorphism $\varphi_A$ or that $A$ and $\varphi_A$ are agreed.

So, by the given $A$-divergent form elliptic operator $L_A$ defined in a domain $\Omega\subset\mathbb C$ we can construct so-called a $A$-quasiconformal mapping $\varphi_A:\Omega \to \widetilde{\Omega}$ with a quasiconformal coefficient
$$
K_A=\frac{1+\|\mu\mid L^{\infty}(\Omega)\|}{1-\|\mu\mid L^{\infty}(\Omega)\|},
$$
where $\mu$ defined by (\ref{ComDil}).

Note that the inverse mapping to the $A$-quasiconformal mapping $\varphi_A: \Omega \to \widetilde{\Omega}$ is the $A^{-1}$-quasiconformal mapping \cite{GPU2020}.

\begin{theorem} (\cite{GPU2020}) \label{th2.1}
Let $\Omega,\widetilde{\Omega}$ be domains in $\mathbb R^2$. Then a homeomorphism $\psi : \Omega \to \widetilde{\Omega}$ is an $A$-quasiconformal mapping if and only if $\psi$ generates by the composition rule
$\psi^{*}(\widetilde{u})= \widetilde{u} \circ \psi$ an isometry of Sobolev spaces $L^{1,2}_A(\Omega)$ and $L^{1,2}(\widetilde{\Omega})$, i.e.
$$
\|\psi^{*}(\widetilde{u}) \mid L^{1,2}_A(\Omega)\| = \|\widetilde{u} \mid L^{1,2}(\widetilde{\Omega})\|
$$
for any $\widetilde{u} \in L^{1,2}(\widetilde{\Omega})$.
\end{theorem}

This theorem generalizes the well known property of conformal mappings to generate the isometry of uniform Sobolev spaces $L^{1,2}(\Omega)$ and $L^{1,2}(\widetilde{\Omega})$ (see, for example, \cite{C50}) and refines (in the case $n=2$) the functional characterization of quasiconformal mappings in the terms of isomorphisms of uniform Sobolev spaces \cite{VG75}.

\section{Spectral estimates}

The theory of extension operators on Sobolev spaces permit us to obtain lower estimates of the first non-trivial Neumann eigenvalues of the divergence form elliptic operators $-\textrm{div} [A(z) \nabla u(z)]$ in Ahlfors domains (quasidiscs) $\Omega \subset \mathbb R^2$.

%Recall that a domain $\Omega$ is called a Ahlfors domain ($K$-quasidisc) if it is the image of the unit disc $\mathbb D$ under a %$K$-quasicon\-for\-mal homeomorphism of the plane $\Omega \subset \mathbb R^2$ onto itself. Ahlfors domain represent large class domains %that includes some fractal domains (like snowflakes).

\vskip 0.3cm
\noindent
{\bf Theorem C.}
\textit{ Let $\Omega\subset \mathbb R^2$ and $\widetilde{\Omega}\subset \mathbb R^2$ be Ahlfors domains. Suppose $\widetilde{\Omega} \supset \Omega$, then
$$
\mu_1(A,\Omega) \geq \frac{\mu_1(A,\widetilde{\Omega})}{\|E_{\Omega}\|^2},
$$
where $\|E_{\Omega}\|$ denotes the norm of the linear continuous extension operator
$$
E_{\Omega}: L^{1,2}_A(\Omega)\to L^{1,2}_A(\widetilde{\Omega}).
$$
}

\begin{proof}
Since $\Omega$ is the Ahlfors domain and the matrix A satisfies to the uniform ellipticity condition
there exists the linear extension operator \cite{VGL79}
\begin{equation}\label{ExtOper}
E_{\Omega}:L^{1,2}_A(\Omega) \to L^{1,2}_A(\widetilde{\Omega})
\end{equation}
defined by formula
\[ (E_{\Omega}u)(z) = \begin{cases}
u(z) & \text{if $z \in \Omega$,} \\
\widetilde{u}(z) & \text{if $z \in \widetilde{\Omega} \setminus \overline{\Omega}$,}
\end{cases} \]
where $\widetilde{u}:\widetilde{\Omega} \setminus \overline{\Omega} \to \mathbb R$ is an extension of the function $u$.

Hence for every function $u \in W^{1,2}_A(\Omega)$ we have
\begin{multline}\label{ineq1}
\|u-u_{\Omega} \mid L^2(\Omega)\| \\
{} =\inf\limits_{c\in \mathbb R}\|u-c\mid L^2(\Omega)\|=\inf\limits_{c\in \mathbb R}\|E_{\Omega}u-c\mid L^2(\Omega)\| \\
{} \leq \|E_{\Omega}u-(E_{\Omega}u)_{\widetilde{\Omega}}\mid L^2(\Omega)\| \leq \|E_{\Omega}u-(E_{\Omega}u)_{\widetilde{\Omega}} \mid L^2(\widetilde{\Omega})\|.
\end{multline}

Here $u_{\Omega}$ and $(E_{\Omega}u)_{\widetilde{\Omega}}$ are mean values of corresponding functions $u$ and $E_{\Omega}u$.

Given the Sobolev-Poincar\'e inequality in Ahlfors domains (see, for example, \cite{GU16})
$$
\inf\limits _{c \in \mathbb R} \|E_{\Omega}u-c \mid L^2(\widetilde{\Omega})\| \leq B_{2,2}(A,\widetilde{\Omega}) \|E_{\Omega}u \mid L^{1,2}_A(\widetilde{\Omega})\|, \quad u \in W^{1,2}_A(\widetilde{\Omega}),
$$
and also the continuity of the extension operator \eqref{ExtOper}, i.~e.,
$$
\|E_{\Omega}u \mid L^{1,2}_A(\widetilde{\Omega})\| \leq \|E_{\Omega}\| \cdot \|u \mid L^{1,2}_A(\Omega)\|, \quad \|E_{\Omega}\|< \infty,
$$
we obtain
\begin{multline}\label{ineq2}
\|E_{\Omega}u-(E_{\Omega}u)_{\widetilde{\Omega}} \mid L^2(\widetilde{\Omega})\| \\
\leq B_{2,2}(A,\widetilde{\Omega}) \cdot \|E_{\Omega}u \mid L^{1,2}_A(\widetilde{\Omega})\|
\leq B_{2,2}(A,\widetilde{\Omega}) \cdot \|E_{\Omega}\| \cdot \|u \mid L^{1,2}_A(\Omega)\|.
\end{multline}
Combining inequalities \eqref{ineq1} and \eqref{ineq2} we have
$$
\|u-u_{\Omega} \mid L^2(\Omega)\| \leq B_{2,2}(A,\Omega) \cdot \|u \mid L^{1,2}_A(\Omega)\|,
$$
where
$$
B_{2,2}(A,\Omega) \leq B_{2,2}(A,\widetilde{\Omega}) \cdot \|E_{\Omega}\|.
$$
According to the Min-Max Principle \cite{D95,M}, $\mu_1(A,\Omega)^{-1}=B_{2,2}^2(A,\Omega)$.
Thus, we have
$$
\mu_1(A,\Omega) \geq \frac{\mu_1(A,\widetilde{\Omega})}{\|E_{\Omega}\|^2}.
$$
\end{proof}

From this theorem, using inequalities~\eqref{UEC} about the uniform ellipticity of the matrix $A$ we obtain

\vskip 0.3cm
\noindent
{\bf Corollary D.}
\textit{ Let $\Omega\subset \mathbb R^2$ and $\widetilde{\Omega}\subset \mathbb R^2$ be Ahlfors domains. Suppose $\widetilde{\Omega} \supset \Omega$, then
$$
\mu_1(A,\Omega) \geq \frac{\mu_1(\widetilde{\Omega})}{K\|E\|^2},
$$
where $\mu_1(\widetilde{\Omega})$ is the first non-trivial Neumann eigenvalues of Laplacian in $\widetilde{\Omega}$, $\|E_{\Omega}\|$ denoted the norm of the linear continuous extension operator
$$
E: L^{1,2}_A(\Omega)\to L^{1,2}_A(\widetilde{\Omega}).
$$
}

\vskip 0.3cm

Let us present the construction of the extension operator for domains $\mathbb{D}_A$ that are analogs of the unit disc for the Riemannian metric induced by the matrix $A$. As in the classical case we will use  quasiconformal reflections and Theorem~\ref{th2.1}.

Since there is a M\"obius transformation of the unit disc $\mathbb D \subset \mathbb R^2$ onto the upper halfplane $H^{+}$ we can replace in the definition of domains $\mathbb{D}_A$  the unit disc
$\mathbb D$ to the upper halfplane $H^{+}$.

Consider the following diagram for $A$-quasiconformal mappings
\[
\begin{CD}
L^{1,2}_{A}(\mathbb{D}_A) @>(\varphi^{-1})^{*}>> L^{1,2}(H^{+})\\
@V{\widetilde{\varphi}^{*}\circ\omega\circ(\varphi^{-1})^{*}}VV  @VV{\omega}V \\
L_A^{1,2}(\mathbb R^2) @<\widetilde{\varphi}^{*}<< L^{1,2}(\mathbb R^2)
\end{CD}
\]
where $\omega$ is a symmetry with respect to the real axis,
that  extend any function $u \in L^{1,2}(H^{+})$ to a function $\widetilde{u}$ from $L^{1,2}(\mathbb R^2)$.
In this case $\left\| \omega \right\|=2$.

In accordance to this diagram we define the extension operator on Sobolev spaces
$$
E:L^{1,2}_A(\Omega) \to L_A^{1,2}(\mathbb R^2)
$$
by formula
\[ (Eu)(z) = \begin{cases}
u(z) & \text{if $z \in \Omega$,} \\
\widetilde{u}(z) & \text{if $z \in \mathbb R^2 \setminus \overline{\Omega}$,}
\end{cases}
\]
where $\widetilde{u}:\mathbb R^2 \setminus \overline{\Omega} \to \mathbb R$ is defined as $\widetilde{u}=\varphi^{*}\circ\omega\circ(\varphi^{-1})^{*}u$.

By Theorem~\ref{th2.1} operators $\varphi_A^{*}$ and $(\varphi_A^{-1})^{*}$ are isometries and $\left\| \omega \right\|=2$ then $\left\| E \right\|=2$.

Let $D$ be a smallest disc that contains $\mathbb{D}_A$. Putting $\widetilde{\Omega}=D$ in Corollary~D and taking into account the estimate of the norm of extension operators in $\mathbb{D}_A$ we obtain the following result:
\vskip 0.3cm
\noindent
{\bf Theorem A.}
\textit{Let $\mathbb D_A$ be an $A$-quasidisc. Then the following inequality holds
\begin{equation*}
%\label{Est}
\mu_1(A,\mathbb{D}_A) \geq  \frac{1}{4K} \cdot \left(\frac{j'_{1,1}}{R_{\mathbb{D}_A}}\right)^2,
\end{equation*}
where $K$ is the ellipticity constant, $R_{\mathbb{D}_A}=\max\limits_{\mathbb{D}(0,1)}|\varphi^{-1}(x)-\varphi^{-1}(0)|$ and $j'_{1,1}\approx 1.84118$ denotes the first positive zero of the derivative of the Bessel function $J_1$.
}
\vskip 0.3cm
Now we give estimates for the radius $R_{\mathbb{D}_A}$ of the smallest disc that contains  $\mathbb{D}_A$. This result is a partial solution of the smallest-circle problem (initially suggested by J.~J. Sylvester in 1857).

\begin{theorem}
\label{sil}
Let $\varphi : \mathbb R^2 \to \mathbb R^2$ be an $A$-quasiconformal mapping. Then the following estimate
$$
R_{\mathbb{D}_A}\leq \frac{1}{16} e^{\pi K_A}
$$
holds, where $K_A$ is the quasiconformality coefficient of the mapping $\varphi.$
\end{theorem}

\begin{proof}
Let $r_{\mathbb{D}_A}=\min\limits_{\mathbb{D}(0,1)}|\varphi^{-1}(x)-\varphi^{-1}(0)|$. Then by the global distortion theorem \cite{Ma} we have
$$
R_{\mathbb{D}_A}\leq \frac{1}{16} e^{\pi K_A}r_{\mathbb{D}_A}.
$$
Because $|\mathbb{D}(\varphi(0),r_{\mathbb{D}_A}|\leq|\mathbb{D}_A|=\pi$ we obtain an upper estimate for
$r_{\mathbb{D}_A}\leq 1$.
Therefore
$$
R_{\mathbb{D}_A}\leq \frac{1}{16} e^{\pi K_A}.
$$
\end{proof}

Now recall that
$$
K_A=\frac{1+\|\mu\mid L^{\infty}(\Omega)\|}{1-\|\mu\mid L^{\infty}(\Omega)\|},
$$
where
$$
\mu(z)=\frac{a_{22}(z)-a_{11}(z)-2ia_{12}(z)}{\det(I+A(z))},\quad I= \begin{pmatrix} 1 & 0 \\ 0 & 1 \end{pmatrix}.
$$
It means that $K_A$ is a function of the matrix $A$ only.

Hence, combining Theorem~A and Theorem~\ref{sil} we obtain a connection between principal frequencies of free
non-homogeneous membranes and the Sylvester smallest-circle problem (1857).

\begin{theorem}
Let $\mathbb D_A$ be an $A$-quasidisc. Then the following inequality holds
\begin{equation*}
\mu_1(A,\mathbb{D}_A) \geq  \frac{64 (j'_{1,1})^2}{K}  \exp \left\{ -2\pi \frac{1+\|\mu\mid L^{\infty}(\Omega)\|}{1-\|\mu\mid L^{\infty}(\Omega)\|} \right\},
\end{equation*}
where $K$ is the ellipticity constant and  $j'_{1,1}\approx 1.84118$ denotes the first positive zero of the derivative of the Bessel function
$J_1$.
\end{theorem}

Because $|\mu(z)|\leq \frac{K-1}{K+1}$ we have by the direct calculations
$$
\frac{1+\|\mu\mid L^{\infty}(\Omega)\|}{1-\|\mu\mid L^{\infty}(\Omega)\|} \leq K.
$$

Hence we obtain:

\vskip 0.3cm
\noindent
{\bf Theorem B.}
\textit{Let $\mathbb D_A$ be an $A$-quasidisc. Then the following inequality holds
\begin{equation*}
\mu_1(A,\mathbb{D}_A) \geq  \frac{ 64 (j'_{1,1})^2}{K \exp \{2\pi K\}},
\end{equation*}
where $K$ is the ellipticity constant and  $j'_{1,1}\approx 1.84118$ denotes the first positive zero of the derivative of the Bessel function $J_1$.
}

This theorem gives estimates of principal frequencies of free non-homogeneous membranes in terms of the quasiconformal geometry of domains \cite{AIM}.
\vskip 0.3cm

\textbf{Acknowledgments.} The first author was supported by the United States-Israel Binational Science Foundation (BSF Grant No. 2014055). The second author was supported by RSF Grant No. 20-71-00037 (Section 3).



\vskip 0.3cm

\noindent
Vladimir Gol'dshtein: Department of Mathematics, Ben-Gurion University of the Negev, P.O.Box 653, Beer Sheva, 8410501, Israel

\noindent
\emph{E-mail address:} \email{vladimir@math.bgu.ac.il} \\

\noindent
Valerii Pchelintsev: Division for Mathematics and Computer Sciences, Tomsk Polytechnic University,
634050 Tomsk, Lenin Ave. 30, Russia \\
Department of Mathematical Analysis and Theory of Functions,
Tomsk State University, 634050 Tomsk, Lenin Ave. 36, Russia

\noindent							
\emph{E-mail address:} \email{vpchelintsev@vtomske.ru}   \\

\noindent			
Alexander Ukhlov:	Department of Mathematics, Ben-Gurion University of the Negev, P.O.Box 653, Beer Sheva, 8410501, Israel

\noindent							
\emph{E-mail address:} \email{ukhlov@math.bgu.ac.il

\end{document}